\documentclass[11pt,reqno,a4paper]{amsart} %(principal)
\usepackage{graphicx}
\usepackage{amssymb,amsmath}
\usepackage{amsthm}
\usepackage{color,graphicx}
\usepackage{hyperref}
\usepackage{color}
\usepackage{mathabx}
\usepackage{subfigure}

\usepackage{subfigure} 

\usepackage{verbatim} %serve para comentar em blocos, usando \begin{comment}

\usepackage{relsize} %tornar simbolos matemáticos maiores (usando \mathlarger{} )

\newcommand{\be}{\begin{equation}}

\newcommand{\ee}{\end{equation}}
\newtheorem{teo}{Theorem}[section]

\numberwithin{equation}{section}
\numberwithin{figure}{section}

\newtheorem{theorem}{Theorem}[section]
\newtheorem{proposition}[theorem]{Proposition}
\newtheorem{remark}[theorem]{Remark}
\newtheorem{lemma}[theorem]{Lemma}

\newtheorem{definition}[theorem]{Definition}

\begin{document}
\vglue-1cm \hskip1cm
\title[transverse instability for the Schrödinger equation]{Transverse instability of periodic standing waves for the generalized nonlinear Schr\"odinger equation}

\begin{center}



\maketitle

%
%{\bf F\'abio Natali}
%
%{Departamento de Matem\'atica - Universidade Estadual de Maring\'a\\
%Avenida Colombo, 5790, CEP 87020-900, Maring\'a, PR, Brazil.}\\
%{ fmanatali@uem.br}
{\bf Gabriel E. Bittencourt Moraes$^1$ \;  Fábio Natali$^2$ }

{Departamento de Matem\'atica - Universidade Estadual de Maring\'a\\
	Avenida Colombo, 5790, CEP 87020-900, Maring\'a, PR, Brazil.}\\
{ $^1$pg54546@uem.br \\ $^2$fmanatali@uem.br}

%{\bf Fábio Natali}

%{Departamento de Matem\'atica - Universidade Estadual de Maring\'a\\
%	Avenida Colombo, 5790, CEP 87020-900, Maring\'a, PR, Brazil.}\\
%{ fmanatali@uem.br}

\vspace{3mm}

\end{center}

\begin{abstract}
In this paper, we determine the transverse instability of periodic standing wave solutions for the generalized Schr\"odinger equation with fractional power nonlinearity. The existence of periodic waves is determined by using a constrained minimization problem in the complex setting, and it is shown that the corresponding real solution, depending on the power nonlinearity, is always positive or changes its sign. The transverse instability results are then determined by applying the main result given in \cite{RoussetTzvetkov} for the periodic case. 
\end{abstract}

\section{Introduction} 
Consider the nonlinear Schr\"odinger equation (NLS)
\begin{equation}\label{NLS-equation}
	i u_t + \Delta u + |u|^\alpha u = 0,
\end{equation}
where $u:\mathbb{T}_L \times \mathbb{R}\times \mathbb{R}\rightarrow\mathbb{C}$, $\alpha>0$ and $\mathbb{T}_L$ is called the $L$-torus. In our paper, functions defined in the set $\mathbb{T}_L$ satisfy $h(x+L)=h(x)$ for all $x\in\mathbb{R}$.

Transverse instability of nonlinear Schrödinger equation has been studied by some authors. Indeed, in  \cite{RoussetTzvetkov2008} and \cite{RoussetTzvetkov2009} the authors proved the transverse instability for the cubic nonlinear Schr\"odinger equation on $\mathbb{R}^2$ and $\mathbb{R} \times \mathbb{T}_L$, respectively. In \cite{Yamazaki-System}, the author used similar ideas as in \cite{RoussetTzvetkov2009} to study instability for standing waves of a system of nonlinear Schrödinger equations posed on $\mathbb{R} \times \mathbb{T}_L$ for $L>0$ in a convenient open interval. Moreover, in \cite{Yamazaki-Potential} the author studied transverse instability results associated to the equation 
\begin{equation}\label{yamazaki}
    i \partial_t u = - \Delta u + V(x) u - |u|^{p-1} u, \; \; (x,y,t) \in \mathbb{T}_L \times \mathbb{R} \times \mathbb{R},
\end{equation}
where $p>1$ and  $V:\mathbb{R} \rightarrow \mathbb{R}$ is a smooth potential.  The influence of the potential $V$ that appears in $(\ref{yamazaki})$ affects the transverse instability result of the standing wave solution $u(x,y,t) = e^{i\omega t}\tilde{\varphi}_\omega(x,y)$ in the following sense: there exist two critical values $\omega_{*,1} > \omega_{*,0}>0$ such that for $\omega \in (\omega_{*,0}, \omega_{*,1})$ such that for $0 < L \leq (\lambda_\omega)^{-\frac{1}{2}}$, the standing wave is stable. If $(\lambda_\omega)^{-\frac{1}{2}} < L$, the standing wave is unstable. Here $-\lambda_{\omega}<0$ indicates the first negative eigenvalue of a convenient linear operator. In \cite{Yamazaki-NullPotential} the same author studied the transversal stability of standing waves solutions for the equation \eqref{yamazaki} with null potential $V=0$.

Concerning the case $\mathbb{T}_{L_1}\times \mathbb{T}_{L_2}$ the authors in \cite{HakkaevStanislavovaStefanov} studied the transverse instability of standing waves of the form $u(x,y,t)=e^{i\omega t}\tilde{\varphi}(x)$ for the nonlinear Schr\"odinger equation $(\ref{NLS-equation})$
in the case where $\alpha=1$ (quadratic nonlinearity) and $\alpha=2$ (cubic nonlinearity). In both specific cases, it is well known that explicit periodic solutions $\tilde{\varphi}$ depending on the Jacobi elliptic functions are solutions of the nonlinear equation
\begin{equation}\label{eqhakka}
-\tilde{\varphi}''+\omega\tilde{\varphi}-|\tilde{\varphi}|^{\alpha}\tilde{\varphi}=0,
\end{equation}
with $\alpha=1$ or $\alpha=2$. They obtained the existence of a large period $L_{2}$ depending on the periodic solution $\tilde{\varphi}$ such that the standing wave is transversally unstable with respect to perturbations of the same period. The fact that solutions to $(\ref{eqhakka})$ are given explicitly in the cases of $\alpha=1$ and $\alpha=2$ is crucial to obtain  the transversal instability result.

%The motivation for our study is the work in \cite{NataliZK1}, where the author obtained a result of transverse instability of periodic traveling wave
%solutions of the generalized Zakharov-Kuznetsov equation
%\begin{equation}\label{ZK1}
 %   u_t + (\alpha+1)u^{\alpha} u_x + (\Delta u)_x = 0,
%$\end{equation} where $u:\mathbb{T}_L\times\mathbb{R}\times\mathbb{R}\rightarrow\mathbb{R}$ and $\alpha>0$. The author considered periodic traveling wave solutions propagating in one direction and having the form $u(x,y,t) = \varphi(x- ct)$, $c>0$. In this case, $\varphi$ arises as a solution of the following equation
%\begin{equation}\label{travZK}
 %   -\varphi'' + c \varphi - \varphi^{\alpha+1} = 0.
%\end{equation} Depending on the value of $\alpha>0$ we can obtain positive periodic solutions that solve $(\ref{travZK})$ (for all $\alpha>0$) or even periodic waves that change their sign (for $\alpha>0$ even integer). It has been showed in \cite{NataliZK1} that all positive and periodic travelling waves of the form $u(x,y,t) = \varphi(x- ct)$ are transversally unstable. In addition, for specific values of $c$, periodic waves that change their sign are also transversally unstable in the cases $\alpha=2$ and $\alpha=4$.

We shall describe our results. In equation \eqref{NLS-equation}, we can consider periodic standing waves $u(x,y,t) = e^{i\omega t} \varphi(x)$, where $\omega>0$. Substituting this form into \eqref{NLS-equation}, we obtain
\begin{equation}\label{edo}
 -\varphi'' + \omega \varphi - |\varphi|^\alpha \varphi = 0,
\end{equation}
where $\varphi$ is a real function with minimal period $L>0$. 

Now, we consider the perturbation of $u(x,y,t)$ associated to the equation \eqref{NLS-equation} given by 
\begin{equation}\label{perturbation}
	u(x,y,t) = e^{i \omega t} \left( \varphi(x) + v(x,y,t)\right).
	\end{equation}
Substituting \eqref{perturbation} into \eqref{NLS-equation}, using \eqref{edo} and neglecting the nonlinear terms in $v$, we have that
\begin{equation}\label{NLS-equation1}
	i v_t - \omega v + v_{xx} + v_{yy} + |\varphi|^\alpha v + \alpha |\varphi|^\alpha {\rm Re}(v) = 0.
\end{equation}

Next, let us consider the separation of variables of the form
\begin{equation}\label{u-V}
	u(x,y,t) = e^{\lambda t} \cos(\kappa y) {\rm v}(x).
\end{equation}
Replacing \eqref{u-V} into \eqref{NLS-equation1}, we have
\begin{equation}\label{eq-32}
	i \lambda {\rm v}(x) - \omega {\rm v}(x) + {\rm v}(x)'' - \kappa^2 {\rm v}(x) + |\varphi|^\alpha {\rm v}(x) + \alpha |\varphi|^\alpha {\rm Re}({\rm v}(x)) = 0.
\end{equation}

%follows that
%\begin{equation}
%	-\lambda v_2(x) - \omega v_1(x) - \kappa^2 v_1(x) + v_1''(x) + |\varphi|^\alpha v_1(x) + \alpha |\varphi|^\alpha v_1(x) = 0
%\end{equation}
%and
%\begin{equation}
%	\lambda v_1(x) - \omega v_2(x) - \kappa^2 v_2(x) + v_2''(x) + |\varphi|^\alpha v_2(x) = 0.
%\end{equation}
On the other hand, by considering ${\rm v}(x) = v_1(x) + i v_2(x)$, we obtain by $(\ref{eq-32})$ 
the following spectral problem
$$
	 \left( 
\begin{array}{cc}
	0 & \mathcal{L}_2 + \kappa^2 \\
	-  (\mathcal{L}_1 + \kappa^2)  & 0 \\
	\end{array} 
\right) \left( 
\begin{array}{c}
	v_1 \\
	v_2 
\end{array}\right)=\lambda \left( 
	\begin{array}{c}
		v_1 \\ 
		v_2 
	\end{array}
\right),
$$
where the operators $\mathcal{L}_1, \mathcal{L}_2: H_{per}^2\subset L^2_{per} \rightarrow L^2_{per}$ are typical Hill operators defined as
\begin{equation*}\label{L1-L2}
	\mathcal{L}_1 := -\partial_x^2 + \omega - (\alpha+1)|\varphi|^\alpha \; \; \text{ and } \; \; \mathcal{L}_2:= -\partial_x^2 + \omega - |\varphi|^\alpha.
\end{equation*}
Let us consider $J = \left( \begin{array}{cc} 0 & -1 \\ 1 & 0 \end{array} \right)$ and $\mathbb{L}_{per}^2 = L^2_{per} \times L^2_{per}$. We have that $J^{-1} = J^* = -J$ and for $\left( \begin{array}{c} v_1 \\ v_2 \end{array} \right) = J \left( \begin{array}{c} w_1 \\ w_2 \end{array} \right)$, we get
\begin{equation}\label{Sk}
	\underbrace{ J^{-1} \left( \begin{array}{cc} \mathcal{L}_1 + \kappa^2 & 0 \\ 0 & \mathcal{L}_2 + \kappa^2 \end{array} \right) J}_{:=\mathcal{S}(\kappa)} \left( \begin{array}{c} w_1 \\ w_2 \end{array} \right)=\lambda J \left( \begin{array}{c} w_1 \\ w_2 \end{array} \right),
\end{equation}
that is, for ${\rm w} = \left( \begin{array}{c} w_1 \\ w_2 \end{array} \right)\equiv ( w_1, w_2)$, we obtain the following spectral problem 
\begin{equation}\label{L(k)A(k)}
	\mathcal{S}(\kappa) {\rm w} = \lambda J {\rm w}.
\end{equation}

\begin{definition}\label{defi}
	The periodic wave $\varphi$ is said to be transversally stable if $\sigma(\mathcal{S}(\kappa)) \subset i \mathbb{R}$ for all $\kappa > 0$. Otherwise, if $\sigma(\mathcal{S}(\kappa))$ contains a point $\lambda\in\mathbb{C}$ with ${\rm Re}(\lambda) > 0$ and $U\neq0 $ such that $(\ref{L(k)A(k)})$ is satisfied for some $\kappa>0$, the periodic wave $\varphi$ is said to be transversally unstable.
\end{definition}

From $(\ref{L(k)A(k)})$ and Definition $\ref{defi}$, we can see that the problem of  transverse stability reduces, in fact, to a spectral stability problem. To obtain that the solution $\varphi$ is transversally unstable, we are going to use the main result given by \cite{RoussetTzvetkov} whose main assumptions are:
	\begin{itemize}
 \item [\textbf{(H0)}] The linear operator $\mathcal{S}(\kappa)$ defined in $\mathbb{L}_{per}^2$ with domain in $\mathbb{H}_{per}^2$ is self-adjoint for all $\kappa\in\mathbb{R}$.
 
		\item [\textbf{(H1)}] There exists $K>0$ and $\beta > 0$ such that $\mathcal{S}(\kappa) \geq \beta {\rm Id}$ for $|\kappa| \geq K$;
		
		\item [\textbf{(H2)}] The essential spectrum $\sigma_{ess}(\mathcal{S}(\kappa))$ is included in $[c_\kappa,+\infty)$ with $c_\kappa > 0$ for $\kappa \neq 0$;
		
		\item [\textbf{(H3)}] For every $\kappa_1 \geq \kappa_2 \geq 0$, we have $\mathcal{S}(\kappa_1) \geq \mathcal{S}(\kappa_2)$. In addition, if for some $\kappa > 0$ and ${\rm w} \neq 0$, we have $\mathcal{S}(\kappa) {\rm w} = 0$, then $(\mathcal{S}'(\kappa) {\rm w},{\rm w}) > 0$ (with $\mathcal{S}'(\kappa)$ the derivative of $\mathcal{S}$ with respect to $\kappa$);
		
		\item [\textbf{(H4)}] The spectrum $\sigma(\mathcal{S}(0))$ of $\mathcal{S}(0)$ is under the form $\{-\lambda_0\} \cup I$ where $-\lambda_0 < 0$ is an isolated simple eigenvalue and $I$ is included in $[0,+\infty)$.
	\end{itemize}
Assumptions \textbf{(H0)}-\textbf{(H4)} implies that the wave $\varphi$ is transversally unstable according to the Definition $\ref{defi}$.

In order to verify assumptions \textbf{(H0)}-\textbf{(H4)}, we first need to show the existence of  $L$-periodic solutions $\varphi$ for the equation \eqref{edo}. For even periodic solutions, we need to solve the following constrained minimization problem
\begin{equation}\label{minP1}
    \inf \left\{ \mathcal{B}_\omega(u) := \frac{1}{2} \int_0^L |u_x|^2 + \omega |u|^2 dx : u \in \mathbb{H}^1_{per,even} \text{ and } \int_0^L |u|^{\alpha+2} dx = \tau \right\},
\end{equation}
where $\tau > 0$ is a fixed number and $\mathbb{H}^1_{per,even}$ denotes the space $\mathbb{H}^1_{per} = H^1_{per} \times H^1_{per}$ restricted to the even periodic functions. Furthermore and for $\alpha>0$ even integer, we can show the existence of odd periodic solutions $\varphi$ satisfying equation \eqref{edo} by solving the constrained minimization problem 
\begin{equation}\label{minP2}
    \inf \left\{ \mathcal{B}_\omega(u) := \frac{1}{2} \int_0^L |u_x|^2 + \omega |u|^2 dx : u \in \mathbb{H}^1_{per,odd} \text{ and } \int_0^L |u|^{\alpha+2} dx = \upsilon \right\},
\end{equation}
where $\upsilon > 0$ is a fixed number and $\mathbb{H}^1_{per,odd}$ denotes the space $\mathbb{H}^1_{per}$ restricted to the odd periodic functions.

The second requirement in our paper to obtain the transversal instability concerns a convenient spectral analysis over the operators $\mathcal{S}(\kappa)$ for all $\kappa\geq0$. Tools related to Floquet theory for Hill operators of the form $\mathcal{P} = -\partial_x^2 + Q(x)$ can be used to obtain the requirements outlined in \textbf{(H0)-(H4)}. Here, $Q(x)$ represents a smooth, real, and even periodic potential.

Our results are then established.

\begin{teo}\label{main-teo-evensolution}
    Let $L>0$ and $\alpha > 0$ be fixed. Let $\varphi$ be an even positive periodic solution for the equation \eqref{edo} obtained from the minimization problem $(\ref{minP1})$. There exist $\lambda > 0$, $\kappa >0 $ and ${\rm w} \in \mathbb{H}^2_{per}$ such that the spectral problem  \eqref{L(k)A(k)} is verified. In other words, the positive even periodic standing wave solution $\varphi$ for the equation \eqref{edo} is transversally unstable in the sense of the Definition \ref{defi}.
\end{teo}

\begin{teo}\label{main-teo-oddsolution}
    Let $L>0$ be fixed and consider $\alpha > 0$ a fixed even integer. Let $\varphi$ be a periodic solution that changes its sign for the equation \eqref{edo} obtained from the minimization problem $(\ref{minP2})$. There exist $\lambda > 0$, $\kappa > 0$ and ${\rm w} \in \mathbb{H}^2_{per}$ such that the spectral problem  \eqref{L(k)A(k)} is verified. In other words, the odd periodic standing wave solution $\varphi$ for the equation \eqref{edo} is transversally unstable in the sense of the Definition \ref{defi}.
\end{teo}

Our paper is organized as follows:  in \ref{section-existence}, we establish the existence of $L$-periodic solutions $\varphi$ that solves \eqref{edo}. The spectral analysis for the operator $\mathcal{S}(0)$ is obtained in the Section \ref{section-spectral}. Finally, the transverse instability of periodic standing waves $\varphi$ is shown in Section \ref{section-transverseinstability}.\\

\textbf{Notation.} For $s\geq0$ and $L>0$, the Sobolev space
$H^s_{per}:=H^s_{per}(\mathbb{T}_L)$
consists of all real periodic distributions $f$ such that
$$
\|f\|^2_{H^s_{per}}:= L \sum_{k=-\infty}^{\infty}(1+k^2)^s|\hat{f}(k)|^2 <\infty
$$
where $\hat{f}$ is the periodic Fourier transform of $f$. The space $H^s_{per}$ is a  Hilbert space with the inner product denoted by $(\cdot, \cdot)_{H^s}$. When $s=0$, the space $H^s_{per}$ is isometrically isomorphic to the space  $L_{per}^2$ and will be denoted by $L^2_{per}:=H^0_{per}$ (see, e.g., \cite{Iorio}). The norm and inner product in $L^2_{per}$ will be denoted by $\|\cdot \|_{L^2}$ and $(\cdot, \cdot)_{L^2}$. For a complex function $f=f_1+if_2\equiv(f_1,f_2)$, we denote the space $\mathbb{H}^s_{per} := H^s_{per} \times H^s_{per}$ for all $s \geq 0$. Notation $H_{per,even}^s$ indicates the subspace of $H^s_{per}$ constituted by even periodic functions and $H_{per,odd}^s$ is the subspace of $H^s_{per}$ constituted by odd periodic functions.

%For $s\geq0$, we denote
%$
%H^s_{per,e}:=\{ f \in H^s_{per} \; ; \; f \:\; \text{is an even function}\}.
%$
%Endowed with the norm and inner product in $H^s_{per}$. In addition, we denote the Sobolev space $\mathbb{H}_{per}^s$ concerning the complex function $f=f_1+if_2$ as
%$$
%\mathbb{H}^s_{per}:= H^s_{per} \times H^s_{per}, \quad \mathbb{H}^s_{per,e}:=H^s_{per,e} \times H^s_{per,e} \quad \text{and} \quad \mathbb{L}_{per}^2:= L^2_{per} \times L^2_{per}
%$$
%equipped with their usual norms and scalar products.

%When necessary and since $\mathbb{C}$ can be identified with $\mathbb{R}^2$, notations above can be also used in the complex/vectorial case in the following sense: for $f\in \mathbb{H}_{per}^s$ we have $f=f_1+if_2\equiv (f_1,f_2)$, where $f_i\in H_{per}^s$, $i=1,2$.

%The symbols $\sn(\cdot, k), \dn(\cdot, k)$ and $\cn(\cdot, k)$ represent the Jacobi elliptic functions of \textit{snoidal}, \textit{dnoidal}, and \textit{cnoidal} type, respectively. For $k \in (0, 1)$, ${\rm F}(\phi, k)$ and $\E(\phi, k)$  denote the complete elliptic integrals of the first and second kind, respectively, and  we denote by $\K(k)={\rm F}\left(\frac{\pi}{2},k\right)$ and $\E(k)=\E\left(\frac{\pi}{2},k\right)$, (see \cite{byrd}).

Let $\mathcal{A}$ be a linear operator. We denote ${\rm n}(\mathcal{A})$ and ${\rm z}(\mathcal{A})$ as the number of negative eigenvalues and the dimension of the kernel of $\mathcal{A}$, respectively. Moreover, we denote $\mathcal{A}_{even}$ ($\mathcal{A}_{odd}$) as the operator $\mathcal{A}$ restricted in the even (odd) sector of its domain.

Given $z=\xi+i\varsigma \in \mathbb{C}$, we denote $|z|=\sqrt{\xi^2+\varsigma^2}.$

\section{Existence of periodic waves}\label{section-existence}

In this section, we are going to prove the existence of periodic wave solutions $\varphi$ for the equation \eqref{edo}. First, we obtain an even positive solution $\varphi$. Second, we prove the existence of an odd periodic solution $\varphi$ for the equation \eqref{edo}. 

\subsection{Even positive solutions} Let $L>0$ and $\alpha > 0$ be fixed.  To obtain an even periodic wave that solves \eqref{edo}, we use a variational approach in order to minimize a suitable  constrained functional. Indeed, let $\tau > 0$ be fixed and consider the set
\begin{equation*}
	\mathcal{Y}_\tau := \left\{ u \in \mathbb{H}^1_{per,even};\ \int_{0}^{L} |u|^{\alpha + 2} dx = \tau \right\}.
\end{equation*} 
\indent For $\omega > 0$, we define the functional $\mathcal{B}_\omega: \mathbb{H}^1_{per,even} \rightarrow \mathbb{R}$ given by
\begin{equation*}
	\mathcal{B}_\omega(u) := \frac{1}{2} \int_{0}^L |u_x|^2 + \omega |u|^2 dx.
\end{equation*}
Thus, we have the following result:
\begin{proposition}\label{minimizationproblem}
	Let $L>0$ be fixed and consider $\tau>0$ and $\omega > 0$. The minimization problem
	\begin{equation}\label{min-problem}
		\Gamma_\omega:= \inf_{u \in \mathcal{Y}_\tau} \mathcal{B}_\omega(u)
	\end{equation}
has at least one solution, that is, there exists a complex-valued function $\Phi \in \mathcal{Y}_\tau$ such that $\mathcal{B}_\omega(\Phi) = \Gamma_\omega$. Moreover, $\Phi$ satisfies the equation	
\begin{equation}\label{eq-min-problem}
	-\Phi'' + \omega \Phi - |\Phi|^\alpha \Phi = 0.
\end{equation}
\end{proposition}

\begin{proof}
	First, it is easy to see that the functional $\mathcal{B}_\omega$ induces an equivalent norm $\mathbb{H}^1_{per,even}$, that is, there exist positive constants $c_0$ and $c_1$ such that
	\begin{equation*}
		0 \leq c_0 \|u\|_{\mathbb{H}^1_{per}} \leq \sqrt{2 \mathcal{B}_\omega (u)} \leq c_1 \|u\|_{\mathbb{H}^1_{per}}.
	\end{equation*}

In addition, since $\mathcal{B}_\omega(u) \geq 0$ for all $u \in \mathbb{H}^1_{per,even}$, we have that $\Gamma_\omega \geq 0$. From the smoothness of the functional $\mathcal{B}_\omega$, we may consider a sequence of minimizers $(u_n)_{n \in \mathbb{N}} \subset \mathcal{Y}_\tau$ such that 
\begin{equation}\label{convergenceBw}
	\mathcal{B}_\omega (u_n) \rightarrow \Gamma_\omega.
\end{equation} 

From \eqref{convergenceBw}, we have that the sequence $(\mathcal{B}_\omega(u_n))_{n \in \mathbb{N}}$ is bounded $\mathbb{H}^1_{per,even}$. Since $\mathbb{H}^1_{per,even}$ is a reflexive Hilbert space, there exists $\Phi \in \mathbb{H}^1_{per,even}$ such that (modulus a subsequence) we have
\begin{equation*}
	u_n \rightharpoonup  \Phi \text{ weakly in } \mathbb{H}^1_{per,even}.
\end{equation*}

Using the compact embedding $\mathbb{H}^1_{per,even} \hookrightarrow \mathbb{L}^{\alpha+2}_{per,even}$ for all $\alpha > 0$, we have that 
\begin{equation*}
	u_n \rightarrow \Phi \in \mathbb{L}^{\alpha+2}_{per,even}.
\end{equation*}
Moreover, by using the mean value theorem and H\"older's inequality, we obtain
\begin{equation*}
	\bigg| \int_{0}^L |u_n|^{\alpha+2} - |\Phi|^{\alpha+2} dx \bigg|  \leq \int_{0}^L \left| |u_n|^{\alpha+2} - |\Phi|^{\alpha+2} \right| dx  \leq  2(\alpha+2) \left( \|u_n\|_{\mathbb{L}^{\alpha+2}_{per}}^{\alpha+1} + \|\Phi\|_{\mathbb{L}_{per}}^{\alpha+1} \right) \|u_n - \Phi\|_{\mathbb{L}^{\alpha+2}_{per}},
	\end{equation*}
implying that $\int_{0}^L |\Phi|^{\alpha+2} dx = \tau$, that is, $\Phi \in \mathcal{Y}_\tau$. Since $\mathcal{B}_\omega$ is lower semi-continuous,
\begin{equation*}
	\mathcal{B}_\omega(\Phi) \leq \liminf_{n \rightarrow \infty} \mathcal{B}_\omega(u_n),
\end{equation*} 
that is, $\mathcal{B}_\omega(\Phi) \leq \Gamma_\omega$.	On the other hand, once $\Phi \in \mathcal{Y}_\tau$, we conclude that $\mathcal{B}_\omega(\Phi) \geq \Gamma_\omega$. Therefore, we conclude that $\Phi \in \mathcal{Y}_\tau$ is a minimizer of the functional $\Gamma_\omega$, that is,
\begin{equation*}
	\mathcal{B}_\omega(\Phi) = \Gamma_\omega = \inf_{u \in \mathcal{Y}_\tau} \mathcal{B}_\omega(u).
\end{equation*}

Next, by Lagrange multiplier theorem, there exists a constant $c_2 = \frac{2 \mathcal{B}_\omega(\Phi)}{\tau} > 0$ such that
\begin{equation}\label{eqtiwhc2}
	- \Phi'' + \omega \Phi = c_2 |\Phi|^\alpha \Phi.
\end{equation}

A scaling argument $\Psi = \sqrt[\alpha]{c_2} \Phi$ allows us to choose $c_2 = 1$ in \eqref{eqtiwhc2}. Therefore, the complex function $\Phi$ is a periodic minimizer of the problem \eqref{min-problem} and satisfies the equation \eqref{eq-min-problem}.
\end{proof}

\begin{remark}
	Let $\Phi \in \mathbb{H}^1_{per,even}$ the minimizer obtained in the Proposition \ref{minimizationproblem}. It is easy to check that the function $e^{-i \theta} \Phi$ satisfies $\mathcal{B}_\omega(e^{-i\theta} \Phi) = \Gamma_\omega$ for all $\theta\in\mathbb{R}$. Consequently, $e^{-i\theta} \Phi$ satisfies equation 
 $(\ref{eq-min-problem})$ for all $\theta\in\mathbb{R}$.
\end{remark}

We show that $\Phi$ obtained in Proposition $\ref{minimizationproblem}$ can be considered of the form $\Phi=e^{i\theta_0}\varphi$ for some $\theta_0\in\mathbb{R}$ and $\varphi$ is a real periodic even function. Indeed, since $\Phi = \phi_1 + i\phi_2$ satisfies the equation \eqref{eq-min-problem}, we have that
	\begin{equation}\label{1}
		-\phi_1 '' + \omega \phi_1 - \left( \phi_1^2 + \phi_2^2 \right)^{\frac{\alpha}{2}} \phi_1 = 0,
	\end{equation}
and
	\begin{equation}\label{2}
	-\phi_2 '' + \omega \phi_2 - \left( \phi_1^2 + \phi_2^2 \right)^{\frac{\alpha}{2}} \phi_2 = 0.
\end{equation}

Multiplying the equations \eqref{1} and \eqref{2} by $\phi_2$ and $\phi_1$, respectively, and subtracting both results, we obtain 
\begin{equation*}
	-\phi_1 '' \phi_2 + \phi_2'' \phi_1 = 0,
\end{equation*}
that is,
\begin{equation*}
	-\phi_1 ' \phi_2 + \phi_2 ' \phi_1 = \tilde{c},
\end{equation*}
where $\tilde{c}$ is an integration constant. Being $\phi_1$ and $\phi_2$ even functions, we have that $\tilde{c} = 0$, that is, $-\phi_1' \phi_2 + \phi_2' \phi_1 = 0$, implying that $\phi_1 = r \phi_2$ for some $r \in \mathbb{R}$. Thus, $\Phi = (r+i)\phi_2 = e^{i \theta_0} \sqrt{1+r^2}\phi_2$, where $\theta_0$ is the principal argument of the complex number $r+i$. Therefore, if $\varphi = \sqrt{1+r^2} \phi_2$, we conclude that the minimizer $\Phi$ can be rewritten in the form $\Phi = e^{i \theta_0} \varphi$ for some $\theta_0 \in \mathbb{R}$.

The next step is to show that $\varphi$ is positive. 
In fact, we have to notice that the operator $\mathcal{L}$ can be obtained defining the functional $G(u) = E(u) + \omega F(u)$ where $E$ and $F$ are conserved quantities associated to the equation $(\ref{NLS-equation})$ given by
\begin{equation*}
    E(u) = \frac{1}{2} \int_0^L |u_x|^2 - \frac{2}{\alpha + 2} |u|^{\alpha+2} dx \; \; \text{ and } \; \; F(u) = \frac{1}{2} \int_0^L |u|^2 dx.
\end{equation*}
Then, we have that $G'(\varphi,0) = 0$, that is, $(\varphi,0)$ is a critical point of $G$. In addition, we have that $G''(\varphi,0) = \mathcal{L}$. By Proposition \ref{minimizationproblem} and the min-max principle (see \cite[Theorem XIII.2]{ReedSimon}) that ${\rm n}(\mathcal{L}_{even}) \leq 1$. Moreover, using the equation \eqref{edo}, we get
\begin{equation*}
    (\mathcal{L}_{1,even} \varphi, \varphi)_{L^2_{per}} = (\mathcal{L}_1 \varphi, \varphi)_{L^2_{per}}  = - \alpha \int_0^L |\varphi|^{\alpha+2} dx  = - \alpha \tau < 0,
\end{equation*}
that is, ${\rm n}(\mathcal{L}_{1,even}) \geq 1$ and we conclude in fact ${\rm n}(\mathcal{L}_{even}) = 1$. Moreover, since $\mathcal{L}$ has a diagonal form, we automatically obtain
\begin{equation}\label{Leven-even}
    {\rm n}(\mathcal{L}_{1,even}) = 1 \; \; \text{ and } \; \; {\rm n}(\mathcal{L}_{2,even}) = 0.
\end{equation}
By Krein-Ruttman's theorem, we deduce that $\varphi > 0$ and ${\rm z}(\mathcal{L}_{2,even}) = 1$.

\subsection{Odd solutions} Let $\alpha>0$ be a fixed even integer. Using similar arguments as in Subsection 3.1, we can obtain the existence of an odd solution $\varphi$ that satisfies \eqref{edo}. In fact, for $\upsilon > 0$, let us consider 
\begin{equation*}
	\mathcal{X}_\upsilon := \left\{ u \in \mathbb{H}^1_{per,odd};\ \int_{0}^{L} |u|^{\alpha + 2} dx = \upsilon \right\}.
\end{equation*} 

Define the functional $\mathcal{B}_\omega: \mathbb{H}^1_{per,odd} \rightarrow \mathbb{R}$ given by
\begin{equation*}
	\mathcal{B}_\omega(u) := \frac{1}{2} \int_{0}^L |u_x|^2 + \omega |u|^2 dx,
\end{equation*}
where $\omega>0$. We have the following proposition concerning the existence of odd periodic standing wave solutions for the equation $(\ref{NLS-equation})$. The proof follows by similar arguments as done in Proposition $\ref{minimizationproblem}$.
\begin{proposition}\label{minimizationproblem-odd}
	Let $L>0$ be fixed and consider $\upsilon>0$ and  $\omega > 0$. If $\alpha > 0$ is an even integer, the minimization problem
	\begin{equation*}\label{min-problem-odd}
		\Omega_\omega:= \inf_{u \in \mathcal{X}_\upsilon} \mathcal{B}_\omega(u)
	\end{equation*}
has at least one solution, that is, there exists a complex-valued function $\Psi \in \mathcal{X}_\upsilon$ such that $\mathcal{B}_\omega(\Psi) = \Omega_\omega$. Moreover, $\Psi$ satisfies the equation	
\begin{equation*}\label{eq-min-problem-odd}
	-\Psi'' + \omega \Psi - |\Psi|^\alpha \Psi = 0.
\end{equation*}
\end{proposition}
\begin{flushright}
$\blacksquare$
\end{flushright}
\begin{remark}
    The same arguments used in the end of Subsection 3.1 can be repeated in this case to prove the existence of $\theta_1\in\mathbb{R}$ such that $\Psi = e^{i \theta_1} \varphi$, where $\varphi$ is an odd real solution that satisfies equation \eqref{edo}.
\end{remark}

%\begin{remark}
 %   The condition of $\alpha > 0$ be an even integer number can be seen in more details in \cite[Section 2]{AlvesNatali} using planar analysis since the solutions of \eqref{edo} are contained on the level curves of the energy
   % \begin{equation*}
   %     \mathcal{E}(\varphi, \xi) = \frac{\xi^2}{2} - \omega %\frac{\varphi^2}{2} + \frac{\varphi^{\alpha+2}}{\alpha+2},
 %   \end{equation*}
 %   where $\xi = \varphi'$.
%\end{remark}

\section{Spectral Analysis}\label{section-spectral}
From the definition of $\mathcal{S}(\kappa)$ in $(\ref{Sk})$ and the fact that
\begin{equation}\label{operator-Lcal}
    \mathcal{L}:= \left( \begin{array}{cc} \mathcal{L}_1 & 0 \\ 0 & \mathcal{L}_2 \end{array} \right),
\end{equation}
is a diagonal operator, we immediately obtain ${\rm n}(\mathcal{L})={\rm n}(\mathcal{S}(0))$ and ${\rm z}(\mathcal{L})={\rm z}(\mathcal{S}(0))$.

\subsection{Spectral analysis for even positive periodic solutions} Let $L>0$ and $\alpha > 0$ be fixed. Consider $\varphi$ the even positive solution of \eqref{edo} obtained in Proposition \ref{minimizationproblem}. By \eqref{Leven-even} we obtain that ${\rm n}(\mathcal{L}_{1,even}) = 1$ and ${\rm n}(\mathcal{L}_{2,even}) = 0$. Since $\mathcal{L}_{1,odd}\varphi' = 0$, we can conclude that $\lambda = 0$ is the first eigenvalue of the operator $\mathcal{L}_{1,odd}$ and thus ${\rm n}(\mathcal{L}_{1,odd}) = {\rm n}(\mathcal{L}_{2,odd}) = 0$. Therefore, we have
\begin{equation*}
    {\rm n}(\mathcal{L}) = {\rm n}(\mathcal{L}_{even}) + {\rm n}(\mathcal{L}_{odd}) = 1.
\end{equation*}
On the other hand, since $\varphi$ is positive we can use \cite[Lemma 3.7]{AlvesNatali} to deduce ${\rm z}(\mathcal{L}_1) = 1$. Now, we have that $0$ is a simple eigenvalue associated to the linear operator $\mathcal{L}_2$, and thus
\begin{equation*}
    {\rm Ker}(\mathcal{L}) = \left[ (\varphi',0), (0,\varphi) \right].
\end{equation*}
Summarizing the above, we have the following result:
\begin{proposition}\label{proposition-spectralanalysis-even}
    Let $L>0$ and $\alpha > 0$ be fixed. Consider $\varphi$ as the positive solution of \eqref{edo} obtained in the Proposition \ref{minimizationproblem}. Operator $\mathcal{L}$ defined in \eqref{operator-Lcal} has exactly one simple negative eigenvalue which is simple and zero is a double eigenvalue with eigenfunctions $(\varphi',0)$ and $(0,\varphi)$. Moreover, the remainder of the spectrum is constituted by a discrete set of eigenvalues.
\end{proposition}
\begin{flushright}
$\blacksquare$
\end{flushright}
\subsection{Spectral analysis for odd periodic solutions} Let $L>0$ be fixed and consider $\alpha > 0$ an even integer. Let $\varphi$ be the odd solution of \eqref{edo} obtained in Proposition \ref{minimizationproblem-odd}. The arguments in \cite[Lemma 3.6]{AlvesNatali} enable to conclude that ${\rm n}(\mathcal{L}_1) = 2$ but, according to the assumption \textbf{(H4)}, the linear operator $\mathcal{S}(0)$ cannot have more than one negative eigenvalue. To overcome this problem and since ${\rm n}(\mathcal{L})={\rm n}(\mathcal{S}(0))$, we shall consider the restriction of the operator $\mathcal{L}_{odd}$. First, we see from Krein-Rutman's theorem that the first eigenvalue of $\mathcal{L}_1$ is simple and its associated to a even and positive eigenfunction. Since ${\rm n}(\mathcal{L}_{1}) = {\rm n}(\mathcal{L}_{1,odd}) + {\rm n}(\mathcal{L}_{1,even})$ and $1\leq{\rm n}(\mathcal{L}_{1,even}) \leq 2$, we obtain that ${\rm n}(\mathcal{L}_{1,odd}) \leq 1$. On the other hand, we have
\begin{equation*}
    ( \mathcal{L}_{1,odd} \varphi,\varphi)_{L^2_{per,odd}} = (\mathcal{L}_1 \varphi,\varphi)_{L^2_{per}} = - \int_0^L \alpha |\varphi|^{\alpha+2} = - \upsilon \alpha < 0,
\end{equation*}
so that ${\rm n}(\mathcal{L}_{1,odd}) = 1$.

Let $\lambda_0^{(i)}$ and $\lambda_1^{(i)}$ be the first two simple eigenvalues associated to the linearized operator $\mathcal{L}_{i,odd}$, $i=1,2$. Since $\mathcal{L}_{1,odd} < \mathcal{L}_{2,odd}$, we can use the standard comparison theorem (see \cite[Theorem 2.2.2]{Eastham}) to obtain
\begin{equation*}
    \lambda_0^{(1)} < \lambda_0^{(2)} \; \;  \; \text{ and } \;  \; \; \lambda_1^{(1)} < \lambda_1^{(2)}.
\end{equation*}
The fact that ${\rm n}(\mathcal{L}_{1,odd}) = 1$ implies automatically that $\lambda_0^{(1)} < 0$ and $\lambda_1^{(1)} = 0$. Thus, $\lambda_1^{(2)} > 0$ and since $\mathcal{L}_2 \varphi = 0$, we obtain $\lambda_0^{(2)} =0$, so that ${\rm n}(\mathcal{L}_{2,odd})=0$.

Therefore, since ${\rm n}(\mathcal{L}_{1,odd}) = 1$ and ${\rm n}(\mathcal{L}_{2,odd}) = 0$, we conclude ${\rm n}(\mathcal{L}_{odd})=1$ and we have the following result:

\begin{proposition}\label{proposition-spectralanalysis-odd}
    Let $L>0$ be fixed and consider $\alpha > 0$ an even integer number. If $\varphi$ is the odd solution of the equation \eqref{edo} obtained in the Proposition \ref{minimizationproblem-odd}, then the restricted operator $\mathcal{L}_{odd}$ has exactly one negative eigenvalue which is simple and zero is a simple eigenvalue with eigenfunction $(0,\varphi)$. Moreover, the remainder of the spectrum is constituted by a discrete set of eigenvalues.
\end{proposition}

\begin{flushright}
$\blacksquare$
\end{flushright}

\section{transverse instability}\label{section-transverseinstability}
Here, we are going to prove our main results about transverse instability of periodic standing wave solutions for the NLS equation \eqref{NLS-equation}. To do so, we need to check that all assumptions $\textbf{(H0)}-\textbf{(H4)}$ are verified.

Initially, we can proof the following result with respect to the coercitivity of the operator $\mathcal{S}(\kappa)$ for $|\kappa|$ large enough. Since the linear operator $\mathcal{S}(\kappa)$ has a square term $\kappa^2$, we only consider $\kappa\geq0$ by symmetry.

\begin{lemma}\label{lema-H1}
Let $L>0$ and $\alpha > 0$ be fixed. Consider the wave solution $\varphi$ of the equation \eqref{edo} given by Proposition $\ref{minimizationproblem}$ or $\ref{minimizationproblem-odd}$. There exist $K>0$ and $\beta > 0$ such that $ \mathcal{S}(\kappa) \geq \beta {\rm Id}$ for $\kappa \geq K.$
\end{lemma}

\begin{proof}
Consider the periodic wave solution $\varphi$ of the equation \eqref{edo} given by Proposition $\ref{minimizationproblem}$ or $\ref{minimizationproblem-odd}$. In addition, let $-\lambda_0$ be the first eigenvalue of the operator $\mathcal{L}$ defined in \eqref{operator-Lcal}. We obtain by min-max theorem 
\begin{align*}
    \left( \mathcal{S}(\kappa)(f,g), (f,g) \right)_{\mathbb{L}^2_{per}} & = \left( (\mathcal{L}_2 + \kappa^2) f, f\right)_{L^2_{per}} + \left( (\mathcal{L}_1 + \kappa^2) g,g \right)_{L^2_{per}} \\
    & = (\mathcal{L}_2 f, f)_{L^2_{per}} + (\mathcal{L}_1 g,g )_{L^2_{per}} + \kappa^2 (f,f)_{L^2_{per}} + \kappa^2 (g,g)_{L^2_{per}} \\
    & =  \left(\mathcal{L} (g,f), (g,f) \right)_{\mathbb{L}^2_{per}} + \kappa^2 \left( (f,g), (f,g) \right)_{\mathbb{L}^2_{per}} \\
    & \geq - \lambda_0 \|(f,g)\|_{\mathbb{L}^2_{per}}^2 + \kappa^2 \|(f,g)\|_{\mathbb{L}^2_{per}}^2 \\
    & = (\kappa^2 - \lambda_0) \|(f,g)\|_{\mathbb{L}^2_{per}}^2.
    \end{align*}
Letting $K > \sqrt{\lambda_0}$, we have that $\beta = \kappa^2 - \lambda_0$ is positive for all  $\kappa > K$. 
\end{proof}

%$\textbf{(H2)}$ As we are working in the periodic scenario, the essential spectrum $\sigma_{ess}(L(\kappa))$ is an empty set for all $\kappa$. Thus, the assumption $\textbf{(H2)}$ is clearly obtained for any case of $\alpha$ studied here: $\alpha \in (0,1)$ or $\alpha=2$ or $\alpha=4$.

The next result provides us with information regarding the increase in relation to the operator $\mathcal{S}(\kappa)$ in terms of $\kappa$ and the behavior of the derivative of $\mathcal{S}(\kappa)$ with respect to $\kappa.$

\begin{lemma}\label{lema-H3}
    For every $\kappa_1 \geq \kappa_2 \geq 0$, we have $\mathcal{S}(\kappa_1) \geq \mathcal{S}(\kappa_2)$. In addition, if $\mathcal{S}'(\kappa)$ denotes the derivative of $\mathcal{S}(\kappa)$ with respect to $\kappa$, we have that $(\mathcal{S}'(\kappa) {\rm w},{\rm w}) > 0$ for all ${\rm w} \in \mathbb{H}^2_{per}$ and $\kappa > 0$.
\end{lemma}
\begin{proof}
For $\kappa_1 \geq \kappa_2 \geq 0$, we see that
\begin{equation*}
    \left( ( \mathcal{S}(\kappa_1) - \mathcal{S}(\kappa_2) ) (f,g), (f,g) \right)_{\mathbb{L}^2_{per}} =  (\kappa_1^2 - \kappa_2^2) \|(f,g)\|_{\mathbb{L}^2_{per}}^2 \geq 0,
\end{equation*}
for all $(f,g) \in \mathbb{L}^2_{per}$, that is, $\mathcal{S}(\kappa_1) \geq \mathcal{S}(\kappa_2)$ for all $\kappa_1 \geq \kappa_2 \geq 0$. Moreover, for all $\kappa >0$ we have
\begin{equation*}
    \mathcal{S}'(\kappa) = \left( \begin{array}{cc} 2 \kappa & 0 \\ 0 & 2 \kappa \end{array} \right),
\end{equation*}
so that $(\mathcal{S}'(\kappa) {\rm w}, {\rm w})_{L^2_{per}} > 0$ for all ${\rm w} \in \mathbb{H}^2_{per}$ and $\kappa > 0$.
\end{proof}

Lemmas \ref{lema-H1} and \ref{lema-H3} and the spectral analysis obtained in Section \ref{section-spectral} are useful to obtain the proofs of Theorems $\ref{main-teo-evensolution}$ and $\ref{main-teo-oddsolution}$.\\

\begin{proof}[Proof of theorems \ref{main-teo-evensolution} and \ref{main-teo-oddsolution}]
Let $L>0$ and $\alpha > 0$ be fixed. First of all, we see that condition (\textbf{H0}) is verified since $\mathcal{S}(\kappa)$ is clearly a self-adjoint operator in $\mathbb{L}_{per}^2$  and the reason that  $\mathcal{L}+\kappa^2$ is also a self-adjoint operator in the same space for all $\kappa\in\mathbb{R}$. Consider then $\varphi$ as the even positive periodic solution obtained by Proposition \ref{minimizationproblem}. From the Lemma \ref{lema-H1} we have that (\textbf{H1}) is satisfied. As we are working in the periodic scenario, the essential spectrum $\sigma_{ess}(\mathcal{S}(\kappa))$ is an empty set for all $\kappa\in\mathbb{R}$ and assumption $\textbf{(H2)}$ is established. Condition (\textbf{H3}) is obtained by Lemma \ref{lema-H3}. We also have that condition (\textbf{H4}) is verified by Proposition \ref{proposition-spectralanalysis-even} and the fact that ${\rm n}(\mathcal{L})={\rm n}(\mathcal{S}(0))$.

On the other hand, let $\alpha > 0$ be an even integer. If $\varphi$ is the odd periodic solution obtained in the Proposition \ref{minimizationproblem-odd}, we can use the same arguments as used above for positive solutions to obtain that assumptions (\textbf{H0})-(\textbf{H3}) are easily verified. In addition, by Proposition \ref{proposition-spectralanalysis-odd}, ${\rm n}(\mathcal{L}_{odd})={\rm n}(\mathcal{S}(0)_{odd})=1$ has requested in the assumption (\textbf{H4}). The theorem is now proved.
\end{proof}

\section*{Acknowledgments}
F. Natali is partially supported by CNPq/Brazil (grant 303907/2021-5). G. E. Bittencourt Moraes is supported by Coordenação de Aperfeiçoamento de Pessoal de Nível Superior (CAPES)/Brazil - Finance code 001.

\end{document}